\documentclass[12pt]{amsart}
\usepackage{amssymb,latexsym,amsfonts,amsthm,verbatim,amscd,
            ifthen,amsmath,graphicx,color,xypic,tikz} 
\DeclareMathAlphabet{\mathpzc}{OT1}{pzc}{m}{it}

\begin{document}
\title[Regular CW-complexes and poset resolutions]{Regular CW-complexes and poset resolutions of monomial ideals}
\author{Timothy B.P. Clark and Alexandre Tchernev}
\date{\today}

\newtheorem{intheorem}{Theorem} 
\newtheorem{inlemma}[intheorem]{Lemma} 
\newtheorem{inproposition}[intheorem]{Proposition} 
\newtheorem{incorollary}[intheorem]{Corollary}
\newtheorem{indefinition}[intheorem]{Definition} 
\newtheorem{inremark}[intheorem]{Remark} 
\newtheorem{innotation}[intheorem]{Notation}

\newtheorem{theorem}[equation]{Theorem} 
\newtheorem{lemma}[equation]{Lemma} 
\newtheorem{proposition}[equation]{Proposition} 
\newtheorem{corollary}[equation]{Corollary} 
\newtheorem{observation}[equation]{Observation} 
\newtheorem{conjecture}[equation]{Conjecture} 
\newtheorem{question}[equation]{Question}
\newtheorem{fact}[equation]{Fact} 
\newtheorem{example}[equation]{Example} 
\newtheorem{definition}[equation]{Definition} 
\newtheorem{remark}[equation]{Remark} 
\newtheorem{remarks}[equation]{Remarks} 
\newtheorem{notation}[equation]{Notation} 
\newtheorem*{acknowledgements}{Acknowledgements}

\newenvironment{indented}{\begin{list}{}{}\item[]}{\end{list}} 

\renewcommand{\:}{\! :\ } 
\newcommand{\p}{\mathfrak p} 
\newcommand{\m}{\mathfrak m}
\newcommand{\e}{\epsilon} 
\newcommand{\g}{{\bf g}} 
\newcommand{\lra}{\longrightarrow} 
\newcommand{\ra}{\rightarrow} 
\newcommand{\altref}[1]{{\upshape(\ref{#1})}} 
\newcommand{\bfa}{\mathbf a} 
\newcommand{\bfb}{\boldsymbol{\beta}} 
\newcommand{\bfg}{\boldsymbol{\gamma}} 
\newcommand{\bfd}{\boldsymbol{\delta}} 
\newcommand{\bfM}{\mathbf M} 
\newcommand{\bfN}{\mathbf N}
\newcommand{\bfI}{\mathbf I} 
\newcommand{\bfC}{\mathbf C} 
\newcommand{\bfB}{\mathbf B} 
\newcommand{\mcP}{\mathcal P}
\newcommand{\mcX}{\mathcal X}
\newcommand{\D}{\textbf{\textup{D}}}
\newcommand{\bsfC}{\bold{\mathsf C}} 
\newcommand{\bsfT}{\bold{\mathsf T}}
\newcommand{\mc}{\mathcal} 
\newcommand{\smsm}{\smallsetminus} 
\newcommand{\ol}{\overline} 
\newcommand{\twedge}
           {\smash{\overset{\mbox{}_{\circ}}
                           {\wedge}}\thinspace} 
\newcommand{\mbb}[1]{\mathbb{#1}}
\newcommand{\pring}{\Bbbk[x_1,\ldots,x_n]}
\newcommand{\irr}{(x_1,\ldots,x_d)}
\newcommand{\Z}{\textup{Z}}
\newcommand{\B}{\textup{B}}
\newcommand{\La}{\mathcal{L}}
\newcommand{\redHo}{\widetilde{\textup{H}}}
\newcommand{\redH}[2]{\widetilde{\textup{H}}_{#1}(#2)}
\newcommand{\pr}{\textup{proj}}
\newcommand{\precdot}{\prec\!\!\!\cdot\;}
\newcommand{\succdot}{\;\cdot\!\!\!\succ}
\newcommand{\vset}[2]{\textbf{V}_{{#1},{#2}}}
\newcommand{\sx}[2]{\textbf{B}_{{#1},{#2}}}
\newcommand{\cplx}[2]{\Delta_{{#1},{#2}}}
\newcommand{\vx}[1]{(\varnothing,#1)}
\newcommand{\ds}{\displaystyle}
\newcommand{\Sy}{\Sigma}
\newcommand{\Ho}{\widetilde{H}}
\newcommand{\sgn}{\textup{sgn}}
\newcommand{\CC}{\widetilde{\mathcal{C}}_\bullet}
\newcommand{\ld}{\lessdot}
\newcommand{\mdeg}{\textup{mdeg}}
\newcommand{\id}{\textup{id}}
\newcommand{\sd}{\textbf{sd}}
\newcommand{\bd}{\textup{Bd}}
\newcommand{\bb}{\textbf{b}}
\newcommand{\lcm}{\mathrm{lcm}}
\newcommand{\rk}{\textup{rk}}

\newlength{\wdtha}
\newlength{\wdthb}
\newlength{\wdthc}
\newlength{\wdthd}
\newcommand{\elabel}[1]
           {\label{#1}  
            \setlength{\wdtha}{.4\marginparwidth}
            \settowidth{\wdthb}{\tt\small{#1}} 
            \addtolength{\wdthb}{\wdtha}
            \smash{
            \raisebox{.8\baselineskip}
            {\color{red}
             \hspace*{-\wdthb}\tt\small{#1}\hspace{\wdtha}}}}  

\newcommand{\mlabel}[1] 
           {\label{#1} 
            \setlength{\wdtha}{\textwidth}
            \setlength{\wdthb}{\wdtha} 
            \addtolength{\wdthb}{\marginparsep} 
            \addtolength{\wdthb}{\marginparwidth}
            \setlength{\wdthc}{\marginparwidth}
            \setlength{\wdthd}{\marginparsep}
            \addtolength{\wdtha}{2\wdthc}
            \addtolength{\wdtha}{2\marginparsep} 
            \setlength{\marginparwidth}{\wdtha}
            \setlength{\marginparsep}{-\wdthb} 
            \setlength{\wdtha}{\wdthc} 
            \addtolength{\wdtha}{1.1ex}
            \marginpar{\vspace*{-0.3\baselineskip}
                       \tt\small{#1}\\[-0.4\baselineskip]\rule{\wdtha}{.5pt} }
            \setlength{\marginparwidth}{\wdthc} 
            \setlength{\marginparsep}{\wdthd}  }  
            
\renewcommand{\mlabel}{\label} 
\renewcommand{\elabel}{\label} 

\newcommand{\mysection}[1]
{\section{#1}\setcounter{equation}{0}
             \numberwithin{equation}{section}}

\newcommand{\mysubsection}[1]
{\subsection{#1}\setcounter{equation}{0}
                \numberwithin{equation}{subsection}}

\newcommand{\mysubsubsection}[1]
{\subsubsection{#1}\setcounter{equation}{0}
                \numberwithin{equation}{subsubsection}}

\maketitle                
\begin{abstract}
We use the natural homeomorphism between a regular CW-complex 
$X$ and its face poset $P_X$ to establish a canonical 
isomorphism between the cellular chain complex of $X$ 
and the result of applying the poset construction of \cite{Clark2010} to $P_X$. 
For a monomial ideal whose free resolution 
is supported on a regular CW-complex, this isomorphism 
allows the free resolution of the ideal to be realized 
as a CW-poset resolution.  Conversely, 
any CW-poset resolution of a monomial ideal gives rise 
to a resolution supported on a regular CW-complex. 
\end{abstract}

\section*{Introduction}
A {plethora} of commutative algebra research 
has centered on the search for combinatorial and 
topological objects whose structure can 
be exploited to give explicit constructions of 
free resolutions. Many 
approaches take advantage of the combinatorial data inherent 
in the $\mathbb{Z}^{n}$-grading of a monomial ideal $N$ and 
produce a resolution of the module $R/N$. 
In particular, the search for regular CW-complexes 
which support resolutions has been quite active, 
due to the well-behaved nature 
and variety of construction methods 
for these topological spaces. 
For example, Diana Taylor's resolution 
\cite{Taylor1966} may be viewed as a 
resolution supported on the full simplex whose 
vertices are labeled with the 
minimal generators of $N$. Bayer, Peeva, and Sturmfels 
\cite{BayerPeevaSturmfels1995} 
take an approach which resolves a specific class of 
ideals using a canonical subcomplex of the same full 
simplex. More generally, a criterion for using regular 
CW-complexes to support resolutions is developed in 
\cite{BayerSturmfels1998}. More recently, these 
techniques were connected to discrete Morse theory in 
\cite{BatziesWelker2002}, where a topological approach 
for reducing the length of regular CW resolutions is 
discussed. Furthermore, the individual techniques of 
Visscher \cite{Visscher2006}, 
Sinefakopolous \cite{Sinefakopolous2008}, and 
Mermin \cite{Mermin2010} each use the framework of regular 
CW-complexes to describe resolutions of individual 
classes of ideals. Despite the richness of these results, 
Velasco \cite{Velasco2008} showed that there exist minimal 
free resolutions of monomial ideals 
which cannot be supported on any CW-complex. Velasco's result 
makes it clear that the structural framework provided by 
CW-complexes is too restrictive to encompass the entire spectrum 
of structures needed for supporting minimal free resolutions 
of monomial ideals, and thus a more 
flexible structural framework is needed. 

The main goal of this paper is to provide an additional 
argument that the framework based on posets introduced 
in \cite{Clark2010} (and refered to as the \emph{poset construction}) 
provides the necessary flexibility. Our main 
result, Theorem~\ref{ChainComplexIso}, is that the poset 
construction, when applied to the face poset of a regular 
CW-complex $X$ recovers the cellular chain complex of $X$. 
As demonstrated in \cite{Clark2010} this provides a very useful 
criterion for checking if a given monomial 
ideal is supported on a regular CW-complex.   
The final evidence that the poset construction is the 
right tool to study the structure of minimal free 
resolutions of monomial ideals is given in \cite{ClarkTchernevMinSupp}, 
where it is shown that for every monomial ideal $I$ 
there is a poset such that our poset construction applied 
to that poset supports the minimal free resoluton of $I$.  

The paper is organized as follows. 
In Section~\ref{S:TwoComplexes} we extend the connection 
between the combinatorics and topology of regular 
CW-complexes into the category of complexes of vector 
spaces. The underlying combinatorial structure of 
regular CW-complexes was established by 
Bj\"orner \cite{Bjorner1984} as a way to classify 
those spaces whose poset of cellular incidences 
reflects the topology of the space itself. 
The first evidence that extending 
Bj\"orner's correspondence will  be very useful was 
described by the first author in \cite{Clark2012}. 
Mermin's result is recovered there, using poset 
combinatorics and the correspondence of 
Theorem~\ref{ChainComplexIso}. 
Each of the constructions mentioned above 
\cite{Taylor1966,BayerPeevaSturmfels1995,BayerSturmfels1998,
BatziesWelker2002,Visscher2006,Sinefakopolous2008,Mermin2010,OkazakiYanagawa} 
describe a topological structure for a resolution without 
explicitly using the combinatorial interpretation 
described in this paper. 

In Section~\ref{S:MonomialRes}, we use 
Theorem~\ref{ChainComplexIso} to reinterpret a variety of 
constructions for free resolutions of monomial ideals. 
There, we realize any cellular resolution as a CW-poset 
resolution. Lastly, we show that the general framework 
of poset resolutions serves as a common construction 
method for many classes of monomial ideals whose 
minimal resolutions have each been constructed using 
topological means. The value in using the combinatorial 
side of Bj\"orner's correspondence lies in the ability 
to use CW-posets to classify ideals which admit 
cellular (minimal) resolutions. Indeed, determining 
whether a resolution is supported on a regular CW-complex 
now amounts to determining whether 
the resolution is supported on a CW-poset. 
More generally, the result of Velasco makes clear the 
limitations of a purely topological perspective when 
constructing minimal resolutions. 
As such, we propose a shift in focus toward the more 
general notions of poset combinatorics when constructing 
resolutions, since many of the benefits of 
a topological approach appear naturally in the poset 
construction of \cite{Clark2010}. 

Throughout the paper, the notions of a poset, 
its order complex, and that of an algebraic chain complex are 
assumed to be familiar to the reader. 
When describing a topological property of a poset, we 
are implicitly describing the property of the order 
complex of the poset. 
For a poset $P$, we write $\Delta(P)$ for its order 
complex. When $X$ 
is a regular CW-complex, we write $P_X$ for its face poset. 


\mysection
{The complexes $\mathcal{C}(X)$ and $\mathcal{D}(P_X)$}
\label{S:TwoComplexes}

Recall that a CW-complex $X$ is said to be \emph{regular} 
if the attaching maps 
which define the incidence structure of $X$ are 
homeomorphisms.  The class of 
regular CW-complexes is studied by 
Bj\"orner in \cite{Bjorner1984}, who 
introduces and investigates the following related class of 
posets. 

\begin{definition}
A poset $P$ is called a \emph{CW-poset} if
\begin{enumerate}
\item $P$ has a least element $\hat{0}$,
\item $P$ is nontrivial (has more than one element),
\item For all $x\in{P}\smsm\{\hat{0}\}$, the open 
      interval $(\hat{0},x)$ is homeomorphic to a sphere.
\end{enumerate}
\end{definition}

Bj\"orner gives the following characterization of CW-posets, 
providing a connection between poset combinatorics and the 
class of regular CW-complexes. 

\begin{proposition}\cite[Proposition 3.1]{Bjorner1984}
A poset $P$ is a CW-poset if and only if 
it is isomorphic to the face poset of a regular CW-complex.
\end{proposition}

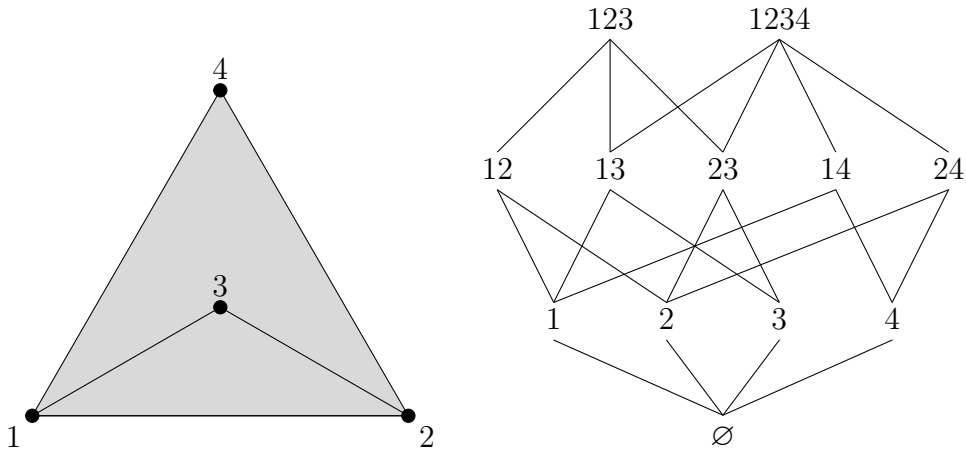
\begin{figure}
\begin{tabular}{cc}
\begin{tikzpicture}[scale=2.5]
\filldraw[fill=white!85!black, draw=black, thin] (0:0) -- (0:2) -- (60:2) -- (0:0) -- cycle;
\draw[draw=black, thin] (0:0) -- (0:2) -- (30:1.1547) -- (0:0) -- cycle;
\draw[fill, black] (0,0) circle (1pt) node[below left, black]{$1$};
\draw[fill, black] (0:2) circle (1pt) node[below right, black]{$2$};
\draw[fill, black] (30:1.1547) circle (1pt) node[above, black]{$3$};
\draw[fill, black] (60:2) circle (1pt) node[above, black]{$4$};
\end{tikzpicture}
& 
\begin{tikzpicture}[scale=0.25]
\draw[fill] (0,-4) circle (0pt) node[below] {$\varnothing$};
\draw[fill] (-9,0) circle (0pt) node[above] {$1$};
\draw[fill] (-3,0) circle (0pt) node[above] {$2$};
\draw[fill] (3,0) circle (0pt) node[above] {$3$};
\draw[fill] (9,0) circle (0pt) node[above] {$4$};
\draw[fill] (-12,8) circle (0pt) node[above] {$12$};
\draw[fill] (-6,8) circle (0pt) node[above] {$13$};
\draw[fill] (0,8) circle (0pt) node[above] {$23$};
\draw[fill] (6,8) circle (0pt) node[above] {$14$};
\draw[fill] (12,8) circle (0pt) node[above] {$24$};
\draw[fill] (-6,16) circle (0pt) node[above] {$123$};
\draw[fill] (3,16) circle (0pt) node[above] {$1234$};
\draw (3,16) -- (-6,10);
\draw (3,16) -- (0,10);
\draw (3,16) -- (6,10);
\draw (3,16) -- (12,10);
\draw (-6,16) -- (-12,10);
\draw (-6,16) -- (-6,10);
\draw (-6,16) -- (0,10);
\draw (-12,8) -- (-9,2);
\draw (-12,8) -- (-3,2);
\draw (-6,8) -- (-9,2);
\draw (-6,8) -- (3,2);
\draw (0,8) -- (-3,2);
\draw (0,8) -- (3,2);
\draw (6,8) -- (-9,2);
\draw (6,8) -- (9,2);
\draw (12,8) -- (9,2);
\draw (12,8) -- (-3,2);
\draw (0,-4) -- (-9,0); 
\draw (0,-4) -- (-3,0); 
\draw (0,-4) -- (3,0); 
\draw (0,-4) -- (9,0); 
\end{tikzpicture}
\end{tabular}
\caption{A two-dimensional regular CW-complex and the Hasse diagram of its face poset}\label{CWPic}
\end{figure}

\begin{remark}
Recall that a ranked poset $P$ has the property that 
for every $x\in P$, all maximal chains with $x$ as 
greatest element have the same finite length. 
Ranked posets admit a rank function and in the case 
of a CW-poset, the rank function $\rk:P_X\ra \mathbb{Z}$ 
takes the form $\rk(\sigma)=\dim(\sigma)+1$. 
Note that the $i$-dimensional cells of a regular 
CW-complex $X$ correspond bijectively to poset elements 
of rank $i+1$ in its face poset $P_X$. 
A two-dimensional example of this correspondence 
is given in Figure \ref{CWPic}. 
\end{remark}

\begin{remark}
Bj\"orner's result arises from the natural (see Lundell and 
Weingram \cite{LundellWeingram1969}) CW-isomorphism 
(we will refer to it in the sequel as 
\emph{Bj\"orner's correspondence}) between a regular 
CW-complex $X$ with collection of closed cells $\{\sigma\}$ 
and the regular CW-complex $\Delta(P_X)$ whose set of closed 
cells is $\{\Delta(\hat{0},\sigma]\}$. Alternately, the space 
$\Delta(P_X)$ may be viewed as a simplicial complex 
whose vertices are indexed by the cells of $X$. 
In fact, $\Delta(P_X)$ is the abstract barycentric 
subdivision of $X$. Furthermore, the simplicial and 
cellular homology groups of $\Delta(P_X)$ (and those of 
subcomplexes corresponding to open intervals of the form 
$(\hat{0},\sigma)$) are isomorphic by applying the 
Acyclic Carrier Theorem (c.f. Theorem 13.4 of 
\cite{Munkres1984}) to the subdivision map $X\ra\Delta(P_X)$. 
\end{remark}

We now describe the algebraic objects that provide 
the setting for the results of this paper. 
Throughout, $X$ will denote a a non-empty, 
finite-dimensional regular CW-complex. 
As is standard, we write $X^i$ for the collection 
of cells of $X$ of dimension $\le i$  
(the $i$-skeleton) and $\mathcal{C}(X)$ for the 
cellular chain complex of $X$ 
with coefficients in the field $\Bbbk$. Precisely, 
$$
\mathcal{C}(X):
0\ra
\mathcal{C}_p
\cdots\ra
\mathcal{C}_i\stackrel{\partial_i}{\lra}
\mathcal{C}_{i-1}\ra\cdots\ra
\mathcal{C}_1\stackrel{\partial_1}{\lra}
\mathcal{C}_0\stackrel{\partial_0}{\lra}
\mathcal{C}_{-1}
$$
with $\mathcal{C}_i=H_i(X^i,X^{i-1},\Bbbk)$, the relative 
homology group in dimension $i\ge 0$.  Since $X$ is nonempty, 
the $(-1)$-skeleton of $X$ consists only of the 
empty cell and by convention, 
we consider $\mathcal{C}_{-1}$ to be a 
one-dimensional vector space 
over $\Bbbk$. More generally, the basis 
elements of the vector spaces appearing 
in the chain complex $\mathcal{C}(X)$ are 
in one-to-one correspondence with the 
cells of $X$. Furthermore, when $\sigma$ 
is a cell of $X$, the differential of 
$\mathcal{C}(X)$ takes the form 
      $$
        \partial(\sigma)=
        \sum_{\substack{\tau\subset\sigma \\ \dim(\sigma)=\dim(\tau)+1}}
        c_{\sigma,\tau}\cdot \tau 
      $$
where the coefficients $c_{\sigma,\tau}\in\{-1,0,1\}$ 
are the so-called \emph{incidence numbers} and are 
determined by the chosen orientations of the 
cells of $X$. See Chapter IV of \cite{Massey1980} 
for details. 

Let $(P,<)$ be a poset with minimum element $\hat{0}$. 
For $p < q\in  P$ where there is no $r\in P$ 
such that $p < r < q$, we say that $q$ \emph{covers} $p$ 
and write $p\lessdot q$. We briefly recall  
the construction of $\mathcal{D}(P)$, a sequence of vector 
spaces and vector space maps, which owes its structure to 
the partial ordering and homology of open intervals in $P$; 
see \cite{Clark2010} for a full description of this 
\emph{poset construction}. In general, $\mathcal{D}(P)$ 
is not necessarily a complex of vector spaces, and in the 
case when $\mathcal{D}(P)$ is indeed is a complex, 
it need not be exact. For the remainder of the paper, 
we restrict discussion of the details 
of this construction as applied to the CW-poset $P_X$. 

The sequence 
$$
\mathcal{D}(P_X):
0\ra
\mathcal{D}_{p+1}\ra 
\cdots\ra
\mathcal{D}_i\stackrel{\varphi_i}{\lra}
\mathcal{D}_{i-1}\ra\cdots\ra
\mathcal{D}_1\stackrel{\varphi_1}{\lra}
\mathcal{D}_0
$$
is constructed using the homology of the order complexes 
of spherical open intervals of the form $(\hat{0},\alpha)$ 
in $P_X$. 
Let $\alpha$ be an element of the CW-poset $P_X$ and write 
$\Delta_\alpha$ as the order complex of the open interval 
$(\hat{0},\alpha)$.
For $i\ge 1$, the vector spaces in $\mathcal{D}(P_X)$ 
take the form 
$$
\mathcal{D}_i=
\bigoplus_{\alpha\in P_X}\redHo_{i-2}(\Delta_\alpha,\Bbbk)= 
\bigoplus_{\rk(\alpha) = i}\redHo_{i-2}(\Delta_\alpha,\Bbbk)
$$ 
(since each $\Delta_\alpha$ is a sphere of dimension 
$\rk(\alpha)-2$), 
and $\mathcal{D}_0=\redHo_{-1}(\{\varnothing\},\Bbbk)$, 
a one-dimensional $\Bbbk$-vector space. 

In order to describe the maps $\varphi_i$ for $i\ge 2$, write $\D_\lambda$ 
for the order complex of a half-closed interval $(\hat{0},\lambda]\subseteq P_X$ and take 
advantage of the decomposition 
$$\Delta_\alpha=\bigcup_{\lambda\ld\alpha}\D_\lambda.$$ 
Next, appeal to the Mayer-Vietoris sequence in reduced homology of the triple 
\begin{equation}\label{MVseq}
  \left(
  \D_\lambda,\hspace{.1in}
  \bigcup_{\underset{\lambda\ne\beta}{\overset{\beta\ld\alpha}{}}}
  \D_\beta,\hspace{.1in}
  \Delta_\alpha
  \right). 
\end{equation}
For notational simplicity when $\lambda\ld\alpha$, let 
$$
\Delta_{\alpha,\lambda}=
\D_{\lambda}\cap\left(\bigcup_{\underset{\lambda\ne\beta}{\overset{\beta\ld\alpha}{}}}
\D_\beta\right).
$$
Write 
$\iota:\Ho_{i-3}(\Delta_{\alpha,\lambda},\Bbbk)\ra\Ho_{i-3}(\Delta_\lambda,\Bbbk)$ 
for the map in homology induced by inclusion and
$$
\delta_{i-2}^{\alpha,\lambda}:
\Ho_{i-2}(\Delta_\alpha,\Bbbk)\ra
\Ho_{i-3}(\Delta_{\alpha,\lambda},\Bbbk)
$$ 
for the connecting homomorphism from the Mayer-Vietoris homology sequence 
of (\ref{MVseq}).  
The map $\varphi_i:\mathcal{D}_i\ra \mathcal{D}_{i-1}$ is defined componentwise as  
$$\varphi_{i}\vert_{\mathcal{D}_{i,\alpha}}=\sum_{\lambda\ld\alpha}\varphi_{i}^{\alpha,\lambda}$$ 
where $\varphi_{i}^{\alpha,\lambda}:\mathcal{D}_{i,\alpha}\ra{\mathcal{D}_{{i-1},\lambda}}$ 
is the composition $\varphi_{i}^{\alpha,\lambda}=\iota\circ\delta_{i-2}^{\alpha,\lambda}$.
In the border case, the map $\varphi_1:\mathcal{D}_1\ra{\mathcal{D}_0}$ is defined 
componentwise as  $$\varphi_1\vert_{\mathcal{D}_{1,\alpha}}=\id_{\Ho_{-1}(\{\varnothing\},\Bbbk)}.$$
We now state the main result of this paper. 

\begin{theorem}\label{ChainComplexIso}
Bj\"orner's correspondence induces a canonical 
isomorphism between the cellular chain complex \
$\mathcal{C}(X)$ and the sequence $\mathcal{D}(P_X)$. 
\end{theorem}

\begin{proof}
Since all homology is taken with coefficients in 
the field $\Bbbk$, we will omit it from the notation. 
By definition of the cellular chain complex we have  
$\mathcal{C}_i = H_i(X^i,X^{i-1})$. 
Furthermore, 
$
H_i(X^i,X^{i-1})=
H_i\left(\widetilde{C}(X^i)/\widetilde{C}(X^{i-1})\right)
$ 
through the realization of the relative homology of 
$(X^i,X^{i-1})$ as the homology of the 
quotient of the cellular chain complex of $X^i$ relative 
to the cellular chain complex of $X^{i-1}$. 
The CW isomorphism between $(X,\{\sigma\})$ and 
$(P_X,\{(\hat{0},\sigma]\})$ induces the isomorphism 
$$
\redHo_i\left(\widetilde{C}(X^i)/\widetilde{C}(X^{i-1})\right)
\cong
\redHo_i\left(\bigoplus_{\rk(\alpha)=i} 
\widetilde{C}({\bf D}_\alpha)/\widetilde{C}(\Delta_\alpha)
\right). 
$$ 
The equality 
$$
\redHo_i\left(\bigoplus_{\rk(\alpha)=i} 
\widetilde{C}({\bf D}_\alpha)/\widetilde{C}(\Delta_\alpha)
\right)
=\bigoplus_{\rk(\alpha)=i} \redHo_i({\bf D}_\alpha,\Delta_\alpha)
$$  
realizes the relative homology of a direct sum of 
quotients of chain complexes as the appropriate 
direct sum of relative homology groups. 
Lastly, we have the isomorphism 
$$
\bigoplus_{\rk(\alpha)=i} \redHo_i({\bf D}_\alpha,\Delta_\alpha)
\cong
\bigoplus_{\rk(\alpha)=i}\redHo_{i-1}(\Delta_\alpha)
=\mathcal{D}_{i+1}, 
$$
which is given by the connecting map in the long exact 
sequence in relative homology 
and is referred to as \emph{reindexing} in the 
Appendix of \cite{Clark2010}. 

Thus, for every $i\ge 0$, we have the following sequence 
of canonical isomorphisms: 
\begin{eqnarray*}
\mathcal{C}_i & = & H_i(X^i,X^{i-1})\\
 & = & H_i\left(\widetilde{C}(X^i)/\widetilde{C}(X^{i-1})\right)\\
 & \cong & H_i\left(\bigoplus_{\rk(\alpha)=i} \widetilde{C}({\bf D}_\alpha)/\widetilde{C}(\Delta_\alpha)\right)\\
& \cong & \bigoplus_{\rk(\alpha)=i}  H_i\left(\widetilde{C}({\bf D}_\alpha)/\widetilde{C}(\Delta_\alpha)\right)\\
 & \cong & \bigoplus_{\rk(\alpha)=i} \redHo_i({\bf D}_\alpha,\Delta_\alpha)\\
 & \cong & \bigoplus_{\rk(\alpha)=i} \redHo_{i-1}(\Delta_\alpha) = \mathcal{D}_{i+1}.
\end{eqnarray*}
The vector spaces $\mathcal{C}_{-1}$ and $\mathcal{D}_0$ 
are each one-dimensional and 
are canonically isomorphic. Therefore, 
the isomorphism between the vector spaces appearing 
in the two chain complexes is established. 

It remains to show that this composition of 
isomorphisms commutes with the differentials 
of the sequences under consideration. Precisely, 
we must show that the following is a commutative diagram.  

\begin{equation}\label{CWposetCD}
\begin{CD} 
\mathcal{C}_i = \redHo_i(X^i,X^{i-1}) @>\partial_i>> \redHo_{i-1}(X^{i-1},X^{i-2})=\mathcal{C}_{i-1} \\
@VVV @VVV \\
\redHo_i\left(\widetilde{C}(X^i)/\widetilde{C}(X^{i-1})\right) 
@>\partial_i>> 
\redHo_{i-1}\left(\widetilde{C}(X^{i-1})/\widetilde{C}(X^{i-2})\right)\\
@VVV @VVV \\
\ds\redHo_i\left(\bigoplus_{\rk(\alpha)=i} \widetilde{C}({\bf D}_\alpha)/\widetilde{C}(\Delta_\alpha)\right) 
@>\partial_i>> 
\ds\redHo_{i-1}\left(\bigoplus_{\rk(\beta)=i-1} \widetilde{C}({\bf D}_\beta)/\widetilde{C}(\Delta_\beta)\right)\\
@VVV @VVV \\
\ds\bigoplus_{\rk(\alpha)=i}  H_i\left(\widetilde{C}({\bf D}_\alpha)/\widetilde{C}(\Delta_\alpha)\right) 
@>\sum_\alpha \partial_i^\alpha>> 
\ds\bigoplus_{\rk(\beta)=i-1}  H_{i-1}\left(\widetilde{C}({\bf D}_\beta)/\widetilde{C}(\Delta_\beta)\right) \\
@VVV @VVV \\
\ds\bigoplus_{\rk(\alpha)=i} \redHo_i({\bf D}_\alpha,\Delta_\alpha) 
@>\sum_\alpha \partial_i^\alpha>> 
\ds\bigoplus_{\rk(\beta)=i-1} \redHo_{i-1}({\bf D}_\beta,\Delta_\beta)\\
@V{\textup{reindex}}VV @VV{\textup{reindex}}V \\
\ds\mathcal{D}_{i+1}=\bigoplus_{\rk(\alpha)=i} \redHo_{i-1}(\Delta_\alpha) 
@>\varphi_{i+1}>> 
\ds\bigoplus_{\rk(\beta)=i-1} \redHo_{i-2}(\Delta_\beta)=\mathcal{D}_{i}\\
\end{CD}
\end{equation}
The commutativity of the bottom square above is immediate 
from Lemma \ref{phiCD}, which is a general fact about 
poset homology.  The commutativity of the remaining  
squares of (\ref{CWposetCD}) is a straightforward 
verification. 
\end{proof} 

Turning to the machinery necessary for Lemma \ref{phiCD}, 
let $j\ge 0$ and write 
$$
\Delta_{\alpha}^{(j)}
=\bigcup_{\substack{\gamma\le\alpha \\ \rk(\gamma)=\rk(\alpha)-j}}
{\bf D}_\gamma 
\quad=\quad \bigcup_{\gamma\ld_j\alpha}
{\bf D}_\gamma  
$$ 
where, to simplify notation, we write 
$\gamma\ld_j\alpha$ when 
$\gamma\le\alpha$ and $\rk(\gamma)=\rk(\alpha)-j$. 
Note that $\Delta_{\alpha}^{(1)}=\Delta_{\alpha}$.
Next, consider the long exact sequence in relative homology 
\begin{equation}\label{Quotients}
\cdots
\ra
\Ho_i(\Delta_{\alpha}^{(j)},\Delta_{\alpha}^{(j+2)},\Bbbk)
\stackrel{\mu}{\ra}
\Ho_i(\Delta_{\alpha}^{(j)},\Delta_{\alpha}^{(j+1)},\Bbbk)
\stackrel{D}{\ra}
\Ho_{i-1}(\Delta_{\alpha}^{(j+1)},\Delta_{\alpha}^{(j+2)},\Bbbk)
\ra
\cdots
\end{equation}
induced by the inclusions 
$$
\Delta_{\alpha}^{(j+2)}\subset 
\Delta_{\alpha}^{(j+1)}\subset \Delta_{\alpha}^{(j)}.
$$

\begin{lemma}\label{phiCD}
Let $P$ be a ranked poset. For each $i\ge 1$ and $\alpha\in P$ the diagram  
\begin{equation*}
\xymatrix{ 
\Ho_{i}(\Delta_{\alpha}^{(j)},\Delta_{\alpha}^{(j+1)},\Bbbk) 
\ar[r]^D \ar@{=}[d]_{\textup{reindex}}
& \Ho_{i-1}(\Delta_{\alpha}^{(j+1)},\Delta_{\alpha}^{(j+2)},\Bbbk) 
\ar@{=}[d]^{\textup{reindex}}\\
\ds\bigoplus_{\rk(\beta)=\rk(\alpha)-j}\Ho_{i-1}(\Delta_{\beta}^{(1)},\Bbbk) \ar[r]^{\varphi_{i+1}}
& \ds\bigoplus_{\rk(\gamma)=\rk(\alpha)-j-1}\Ho_{i-2}(\Delta_{\gamma}^{(1)},\Bbbk)}
\end{equation*}
is commutative. 
\end{lemma}

\begin{proof}
Orient each face $\sigma=\{a^\sigma_0,\cdots,a^\sigma_i\}\in \Delta_\alpha^{(j)}$ using 
the ordering in the chain $a^\sigma_0<\cdots<a^\sigma_i\in P$. 
Next, suppose that $[\bar{w}]$ is a representative 
for the homology class generated by the image 
$\bar{w}\in\widetilde{\mathcal C}_i\left(\Delta_\alpha^{(j)},\Delta_\alpha^{(j+1)}\right)$
of the relative cycle
$$
w
=\sum_{a_i^\sigma\ld_j\alpha} c_\sigma\cdot\sigma
=\sum_{\beta\ld_j\alpha} w_\beta
$$ 
of $\left(\Delta_\alpha^{(j)},\Delta_\alpha^{(j+1)}\right)$,
where $c_\sigma\in\Bbbk$, 
and $\ds w_\beta=\sum_{a_i^\sigma=\beta}c_\sigma\cdot\sigma$. 
Since $w$ is a relative cycle, we must have 
$$
\sum_{a_i^\sigma=\beta} c_\sigma\cdot d(\hat{\sigma})=0
$$ 
where $\hat{\sigma}=\sigma\smsm\{a^\sigma_i\}=\{a^\sigma_0,\cdots,a^\sigma_{i-1}\}$
and $d$ is the simplicial boundary map.
Therefore, each $w_\beta$ is a relative cycle for 
$(\Delta_\beta^{(j)},\Delta_\beta^{(j+1)})$
and reindexing the class 
$
\ds[\bar{w}]=\sum_{\beta)\ld_j\alpha}[\bar{w}_\beta]
$ 
yields  
$$
\sum_{\beta\ld_j\alpha}[d(w_\beta)]
=\sum_{\beta\ld_j\alpha}[v_\beta], $$
where 
$$
v_\beta
=d(w_\beta)
=(-1)^i\sum_{a_i^\sigma=\beta}c_\sigma\cdot\hat{\sigma}
$$
is a cycle in $\Delta_\beta^{(1)}$.
Next, choose for each $\beta\le\alpha$ with 
$\rk\beta=\rk\alpha - j$, a partition 
$$
\left(\hat{0},\beta\right)
=\bigsqcup_{\gamma\ld\beta}P_{\beta,\gamma}
$$
with the property that $\lambda\le\gamma$ for each 
$\lambda\in P_{\beta,\gamma}$. This choice allows us to write 
$$
v_\beta
=
\sum_{\gamma\ld\beta}
w_{\beta,\gamma}
\qquad\text{where}\qquad  
w_{\beta,\gamma}=
(-1)^i\sum_{
\begin{smallmatrix}
 a_i^\sigma=\beta \\ 
 a_{i-1}^\sigma\in P_{\beta,\gamma}
\end{smallmatrix} 
}
 c_\sigma\cdot\hat{\sigma}.
$$
Therefore, for each $\delta<\alpha$ with 
$\rk(\delta)=\rk(\alpha)-j-1$
the component of 
$$
\varphi_{i+1}\left(
\sum_{
\beta\ld_j\alpha 
}
[v_\beta]\right)
$$
in $\Ho_{i-2}(\Delta_{\delta}^{(1)},\Bbbk)$ is given by 
$$
\sum_{\delta\ld\beta\ld_j\alpha} 
   \varphi_{i+1}^{\beta,\delta}\left([v_\beta]\right)=
\sum_{\delta\ld\beta\ld_j\alpha}
   \varphi_{i+1}^{\beta,\delta}
\left(\left[\sum_{\gamma\ld\beta}w_{\beta,\gamma}\right]\right)=
\sum_{\delta\ld\beta\ld_j\alpha}
   \left[d(w_{\beta,\delta})\right].  
$$
On the other hand, since $v_\beta$ is a cycle, 
each $w_{\beta,\gamma}$ is a relative cycle for 
$(\Delta_\alpha^{(j+1)},\Delta_\alpha^{(j+2)})$.
As $D$ is the connecting map in a long exact sequence in 
homology, we have
$$
D([\overline{w}]) 
= 
\left[\overline{d(w)}\right] 
= 
\left[\sum_{
\beta\ld_j\alpha 
}
\bar{v}_\beta\right] 
= 
\sum_{
\beta\ld_j\alpha 
}
\left[
\sum_{\gamma\ld\beta}\overline{w}_{\beta,\gamma}\right], 
$$
which upon reindexing becomes 
$$
\sum_{
\gamma\ld_{j+1}\alpha 
}
\left[d\left(
\sum_{\gamma\ld\beta\ld_j\alpha}
    w_{\beta,\gamma}\right)\right]. 
$$
Therefore, for $\delta < \alpha$ with 
$\rk(\delta)=\rk(\alpha)-j-1$, 
the image in $\Ho_{i-2}(\Delta_{\delta}^{(1)},\Bbbk)$ 
is equal to 
$$
\left[d\left(
\sum_{
\delta\ld\beta\ld_j\alpha  
}
w_{\beta,\delta}\right)\right] 
=\sum_{\delta\ld\beta\ld_j\alpha}
[d(w_{\beta,\delta})],
$$
the desired conclusion. 
\end{proof} 

\mysection
{Free resolutions of monomial ideals}\label{S:MonomialRes}

Let $R=\pring$, considered with its usual $\mathbb{Z}^n$-grading (multigrading) 
and suppose that $N$ is a monomial ideal in $R$. 
A CW-complex $X$ on $r$ vertices inherits a $\mathbb{Z}^n$-grading 
from the $r$ minimal generators of $N$ through the following correspondence. 
Let $\sigma$ be a non-empty cell of $X$, and identify $\sigma$ by its set of 
vertices $V_\sigma$. Set $m_\sigma=\lcm ( m_j \,\vert\, j\in V_\sigma)$, 
so that the multidegree of $\sigma$ is defined as the multidegree of the 
monomial $m_\sigma$. Clearly, the multigrading of $X$ induces 
a multigrading on the face poset $P_X$. 

When $X$ is $\mathbb{Z}^n$-graded, the cellular chain complex 
$\mathcal{C}(X)$ is homogenized in the usual way 
to produce a $\mathbb{Z}^n$-graded chain 
complex of $R$-modules. Precisely, for any cell $\sigma\in X$, 
let $R\sigma$ be a rank one free $R$-module which has 
multidgree $m_\sigma$. The complex $\mathcal{F}_X$ is the 
$\mathbb{Z}^n$-graded $R$-module 
$\ds\bigoplus_{\varnothing\ne\sigma\in X} R\sigma$ 
with differential 
$$   
        \partial(\sigma)=
        \sum_{\substack{\tau\subset\sigma \\ \dim(\sigma)=\dim(\tau)+1}}
        c_{\sigma,\tau}\,\frac{m_\sigma}{m_\tau} \, \tau. 
$$

For $\mathbf{b}\in\mathbb{Z}^n$, write $X_{\preceq \mathbf{b}}$ for the collection of cells in $X$ 
whose multidegrees are comparable to $\mathbf{b}$. Bayer and Sturmfels give the 
following characterization of CW-complexes which support free resolutions. 

\begin{proposition}\cite[Proposition 1.2]{BayerSturmfels1998}
The complex $\mathcal{F}_X$ is a free resolution of $N$ 
if and only if $X_{\preceq \mathbf{b}}$ is acyclic over $\Bbbk$ for 
all multidegrees $\mathbf{b}$. 
\end{proposition}

We now give details of the homogenization of the sequence $\mathcal{D}(P_X)$.  
First, consider the map $\eta:P_X\lra\mbb{Z}^n$ for the map 
induced by the multigrading on $X$. Next, define 
$$
\mathcal{F}(\eta):
  \cdots
  \lra 
  \mathcal{F}_t
  \stackrel{\partial_t}{\lra} 
  \mathcal{F}_{t-1}
  \lra
  \cdots
  \lra
  \mathcal{F}_1
  \stackrel{\partial_1}{\lra} 
  \mathcal{F}_0,
$$
a sequence of free multigraded 
$R$-modules and multigraded $R$-module homomorphisms. 
For $i\ge{1}$, we have 
$$
\ds \mathcal{F}_i=
\bigoplus_{{\hat{0}}\ne\lambda\in P}\mathcal{F}_{i,\lambda}=
\bigoplus_{{\hat{0}}\ne\lambda\in P}R\otimes_{\Bbbk}\mathcal{D}_{i,\lambda}
$$ 
where the multigrading of $x^\bfa\otimes{v}$ 
is defined as $\bfa+\eta(\lambda)$ for any element  
$v\in \mathcal{D}_{i,\lambda}=\redHo_{i-2}(\Delta_\lambda,\Bbbk)$.  

The differential $\partial_i:\mathcal{F}_i\lra \mathcal{F}_{i-1}$ 
in this sequence of multigraded modules is defined on the 
component $\mathcal{F}_{i,\alpha}$ as 
$$
\ds \partial_i\arrowvert_{\mathcal{F}_{i,\alpha}}
=\sum_{\lambda\ld\alpha}\partial^{\alpha,\lambda}_i
$$ 
where 
$\partial^{\alpha,\lambda}_i:\mathcal{F}_{i,\alpha}\lra{\mathcal{F}_{i-1,\lambda}}$ 
takes the form 
$\partial^{\alpha,\lambda}_i=x^{\eta(\alpha)-\eta(\lambda)}\otimes\varphi_{i}^{\alpha,\lambda}$ 
for $\lambda\ld\alpha$.  

Set $F_0=R\otimes_{\Bbbk}\mathcal{D}_0$ and define the multigrading of 
$x^\bfa\otimes{v}$ as $\bfa$ for each $v\in \mathcal{D}_0$.
The differential 
$\partial^{\alpha,\lambda}_1:\mathcal{F}_{1,\alpha}\lra{\mathcal{F}_{0,\lambda}}$ 
is defined componentwise as  
$$
\partial_1\arrowvert_{F_{1,\lambda}}
=x^{\eta(\lambda)}\otimes\varphi_1\arrowvert_{\mathcal{D}_{1,\lambda}}.
$$

Having established a commutative algebra interpretation for 
the objects of Section \ref{S:TwoComplexes}, 
we now state the motivating result of this paper. 

\begin{theorem}\label{CellPosetRes}
Let $N$ be a monomial ideal and suppose that $\mathcal{F}$ is a resolution of $N$.  
The resolution $\mathcal{F}$ is supported on a regular CW-complex $X$ if and only if 
$\mathcal{F}$ is a poset resolution supported on the CW-poset $P_X$.  
\end{theorem}

\begin{proof}
By Theorem \ref{ChainComplexIso}, 
the homogenizations of $\mathcal{D}(P_X)$ and 
$\mathcal{C}(X)$ both produce the same chain complex. 
\end{proof} 

\begin{remark}
Using the language of Peeva and Velasco \cite{PeevaVelasco2011}, 
the chain complexes $\mathcal{D}(P_X)$ and $\mathcal{C}(X)$ 
are isomorphic \emph{frames} for the resolution $\mathcal{F}$. 
\end{remark}

We now apply Theorem \ref{CellPosetRes} to several well-studied resolutions of monomial ideals. 

\begin{corollary}
The Taylor, Scarf, and Lyubeznik resolutions are CW-poset resolutions. 
\end{corollary} 

\begin{proof}
Each of these resolutions are supported on a simplicial complex, 
as described by Mermin in \cite{Mermin2012}. Hence, applying Theorem \ref{CellPosetRes} 
to the face poset of the associated simplicial complex in each case recovers the 
desired CW-poset resolution. 
\end{proof}

Recall that a monomial ideal $N$ is said to be \emph{stable} 
if for every monomial $m\in{N}$, the monomial $m\cdot{x_i}/{x_r}\in{N}$ 
for each $1\le{i}<r$, where $r=\max\{k:x_k\textrm{ divides }m\}$.  
Eliahou and Kervaire \cite{EliahouKervaire1990} first gave a construction 
for the minimal free resolution of a stable ideal, which Mermin \cite{Mermin2010} 
reinterpreted as a resolution supported on a regular CW-complex. 
By using Bj\"orner's correspondence along with Theorem \ref{CellPosetRes}, 
we recover Mermin's result.  

\begin{corollary}
The Eliahou-Kervaire resolution is supported on a regular CW-complex. 
\end{corollary}

\begin{proof}
The main result of \cite{Clark2012} establishes that the Eliahou-Kervaire resolution 
is supported on an EL-shellable CW-poset (of admissible symbols). 
Theorem \ref{CellPosetRes} therefore applies, and the minimal resolution is supported 
on the regular CW-complex whose face poset is isomorphic to this poset of admissible symbols. 
\end{proof}

Another well-studied combinatorial object which can be associated to a monomial ideal 
also serves as a source for CW-posets which support minimal free resolutions. 
Recall that the \emph{lcm-lattice} of a monomial ideal $N$ is 
the set  $L_N$ of least common multiples of the minimal generators of $N$, 
along with $1$ (considered to be the least common 
multiple of the empty set).  Ordering in the lcm-lattice is 
given by $m'<m$ if and only if $m'$ divides $m$.  The homological 
importance of $L_N$ was established by Gasharov, Peeva, and Welker in 
\cite{GasharovPeevaWelker1999}. Motivated by their work, 
the class of lattice-linear ideals was studied by the first author in \cite{Clark2010}. 

\begin{definition}\label{LLDef}
Let $\mathcal{F}$ be a minimal multigraded free resolution of $R/N$.
The monomial ideal $N$ is \emph{lattice-linear} if multigraded
bases $B_k$ of the free modules in $\mathcal{F}$ can be fixed for all $k$ 
so that for any $i\ge 1$ and any $e_m \in{B_i}$ the differential 
$$\partial^{\mathcal{F}}(e_m)=\sum_{e_{m'}\in{B_{i-1}}}c_{m,m'}\cdot{e_{m'}}$$ 
has the property that if the coefficient $c_{m,m'}\ne 0$ 
then $m'\ld m \in{L_N}$.
\end{definition}

With the extension of Bj\"orner's correspondence complete, 
the notion of lattice-linearity may be viewed topologically 
when the lcm-lattice exhibits a precise enough combinatorial structure. 

\begin{corollary}\label{lcmLatticeCW}
If the lcm-lattice of a monomial ideal is a CW-poset then the ideal is 
lattice-linear. Furthermore, such an ideal has a minimal cellular resolution. 
\end{corollary} 

\begin{proof} 
Write $N$ for a monomial ideal and $L_N$ for the CW-poset which is its lcm-lattice. 
Since $L_N$ is a CW-poset, the open interval $(1,m)$ is homeomorphic to 
a sphere for every $m\in L_N$. Theorem 2.1 of \cite{GasharovPeevaWelker1999} 
therefore guarantees that each monomial of $L_N$ corresponds to exactly 
one free module in the minimal free resolution of $N$. 

Aiming for a contradiction, suppose that $N$ is not lattice-linear. Hence, there exists 
a basis element $e_m\in{B_i}$ for some $i$ and a basis element $e_{m'}\in B_{i-1}$ 
for which the coefficient $c_{m,m'}\ne 0$ in the expansion of the differential 
$$\partial^{\mathcal{F}}(e_m)=\sum_{e_{m'}\in{B_{i-1}}}c_{m,m'}\cdot{e_{m'}},$$ 
but $m'$ is not covered by $m$ in $L_N$. Hence, there must exist $\ell\in L_N$ 
for which $m'<\ell<m$. However, the interval $(1,m)$ is homeomorphic to 
an $(i-2)$-sphere, and the interval  $(1,m')$ is homeomorphic to an $(i-3)$-sphere. 
Since $L_N$ is a CW-poset, the open interval $(1,\ell)$ also is homeomorphic to a sphere. 
However, $(1,\ell)$ cannot be homeomorphic to a sphere of dimension $i-2$, nor 
can it be homeomorphic to a sphere of dimension $i-3$. Indeed, the order complex 
of $(1,\ell)$ is a proper subcomplex of the $(i-2)$-dimensional sphere $(1,m)$. Moreover, 
it contains the $(i-3)$-dimensional complex $(1,m')$ as a proper subcomplex. 
Hence, $(1,\ell)$ must be a $j$-dimensional sphere where $i-3<j<i-2$, a contradiction. 
Therefore, such a monomial $\ell$ cannot exist, which means that $N$ is indeed lattice-linear. 

Write $X_N$ for the regular CW-complex 
whose face poset is isomorphic to the lcm-lattice $L_N$. 
The cells of $X_N$ inherit their multidegree directly 
from the corresponding monomial in the lcm-lattice. 
Note that since $L_N$ is a lattice, then $X_N$ 
is a regular CW-complex with unique cell of top dimension, 
which has the property that any pair of cells intersect in a cell. 
Such a CW-complex is said to have the \emph{intersection property}. 
Applying Theorem \ref{CellPosetRes}, the minimal CW-poset resolution 
supported on $L_N$ may be reinterpreted as a minimal 
cellular resolution supported on $X_N$. 
\end{proof} 

\begin{remark}
In fact, more can be said about monomial ideals whose lcm-lattice is a CW-poset. Indeed, 
such ideals are \emph{rigid} in the sense of \cite{ClarkMapes2013}. In general, a rigid ideal 
has a minimal free resolution with unique multigraded basis. In the case 
when the lcm-lattice is a CW-poset, there is a unique regular CW-complex which 
supports the minimal multigraded resolution. The edge ideals of complete bipartite graphs 
studied by Visscher \cite{Visscher2006} are a class of ideals whose lcm-lattice 
is a CW-poset. 
\end{remark} 

In closing, we note that the property of an lcm-lattice being a CW-poset 
is not a necessary condition for lattice-linearity or rigidity. The ideal given by Velasco, 
whose resolution is not supported on any CW-complex \cite{Velasco2008} 
is a lattice-linear, rigid monomial ideal and hence, admits a minimal poset resolution. 
There are many lattice-linear ideals whose lcm-lattice is not a CW-poset. 
The failure of topological methods to completely encode the structure of free resolutions 
allows us to consider the following natural questions, which are the subject of ongoing research. 

\begin{question}
Which CW resolutions supported on a non-regular CW-complex can be realized as poset 
resolutions? Can every resolution be realized as a poset resolution? 
\end{question} 

A positive answer to each of these two questions is given 
respectively in Wood \cite{WoodThesis} and Clark and Tchernev \cite{ClarkTchernevMinSupp}. 
 
\bibliographystyle{amsalpha}
\bibliography{../Bibliography/TBPC-bibliography}

\end{document}